\documentclass[a4paper,12pt,notitlepage]{amsart}
\usepackage{amsmath,amssymb,mathrsfs,amsthm}

\usepackage{verbatim}
\usepackage[spanish,USenglish]{babel} % espanol, ingles
\usepackage{color}
\usepackage{tikz}
\usepackage{graphicx}

\newtheorem{theorem}{Theorem}

\newtheorem{lemma}[theorem]{Lemma}

\newtheorem{proposition}[theorem]{Proposition}
\newtheorem{definition}[theorem]{Definition}

\newtheorem{remark}[theorem]{Remark}

\newcommand{\R}{\mathbb{{R}}}
\newcommand{\Z}{\mathbb{{Z}}}
\newcommand{\N}{\mathbb{{N}}}
\newcommand{\T}{\mathbb{{T}}}
\newcommand{\C}{\mathbb{{C}}}

\begin{document}

\title[A Katznelson-Tzafriri type theorem for Ces\`aro bounded operators]{A Katznelson-Tzafriri type theorem for Ces\`aro bounded operators}

\author[L. Abadias]{Luciano Abadias}
\address{Departamento de Matem\'aticas, Instituto Universitario de Matem\'aticas y Aplicaciones, Universidad de Zaragoza, 50009 Zaragoza, Spain.}
%\curraddr{}
\email{labadias@unizar.es}

%\thanks{The author L. Abadias  has been partially supported by Project MTM2013-42105-P, DGI-FEDER, of the MCYTS, and Project E-64, D.G. Arag\'on.}

% General info


\begin{abstract}
We extend the well-known Katznelson-Tzafriri theorem, originally posed for power-bounded operators, to the case of Ces\`aro bounded operators of any order $\alpha>0.$ For this purpose, we use a functional calculus between a new class of fractional Wiener algebras and the algebra of bounded linear operators, defined for operators with the corresponding Ces\`aro boundedness. Finally, we apply the main theorem to get ergodicity results for the Ces\`aro means of bounded operators.
\end{abstract}

\date{}

\maketitle

\section{Introduction}
\setcounter{theorem}{0}
\setcounter{equation}{0}

Let $A(\mathbb{T})$ be the convolution Wiener algebra formed by all continuous periodic functions $\mathfrak{f}(t)=\sum_{n=-\infty}^{\infty}a(n) e^{int},$ for $t\in [0,2\pi),$ with the norm $\lVert \mathfrak{f} \rVert_{A(\mathbb{T})}:=\sum_{n=-\infty}^{\infty}|a(n)|.$ This algebra is regular. We denote by $A_+(\mathbb{T})$ the convolution closed subalgebra of $A(\mathbb{T})$ where the functions satisfy that $a(n)=0$ for $n<0.$ Note that $A(\mathbb{T})$ and $\ell^1(\Z)$ are isometrically isomorphic. The same holds for $A_+(\mathbb{T})$ and $\ell^1(\N_0),$ where $\N_0=\N\cup\{0\}.$ In the above, the sequence $(a(n))_{n\in\Z}$ corresponds to the Fourier coefficients of $\mathfrak{f},$ that is $$a(n):=\widehat{\mathfrak{f}}(n)=\frac{1}{2\pi}\int_{0}^{2\pi}\mathfrak{f}(t)e^{-int}\,dt.$$ Let $E$ be a closed subset of $\mathbb{T}$ and $\mathfrak{f}\in A(\mathbb{T}).$ We recall that $\mathfrak{f}$ is of spectral synthesis with respect to $E$ if for every $\varepsilon>0$ there exists $\mathfrak{f}_{\varepsilon}\in A(\mathbb{T})$ such that $\lVert \mathfrak{f}-\mathfrak{f}_{\varepsilon} \rVert_{A(\mathbb{T})}<\varepsilon$ with $\mathfrak{f}_{\varepsilon}= 0$ in a neighborhood of $E.$ The above definition is valid in any regular Banach algebra. For more details see \cite[Chapter VIII, Section 7]{Kat}. Since $\sup_{t\in[0,2\pi)}|\mathfrak{f}(t)|\leq \lVert \mathfrak{f}\rVert_{A(\T)},$ if $\mathfrak{f}$ is of spectral synthesis with respect to $E,$ then $f$ vanishes on $E.$

Let $X$ be a complex Banach space and $\mathcal{B}(X)$ the Banach algebra formed by the bounded linear operators on $X$. An operator $T\in\mathcal{B}(X)$ is power-bounded if $\sup_{n\geq 0}\lVert T^n \rVert<\infty.$ In 1986, Y. Katznelson and L. Tzafriri proved that if $T$ is a power-bounded operator on $X$ and $\mathfrak{f}\in A_+(\mathbb{T})$ is of spectral synthesis in $A(\mathbb{T})$ with respect to $\sigma(T)\cap \mathbb{T},$ then $$\displaystyle\lim_{n\to\infty}\lVert T^n\theta(\widehat{\mathfrak{f}}) \rVert=0,$$ where $\sigma(T)$ denotes the spectrum of the operator $T$ and $\theta:\ell^1(\N_0)\to \mathcal{B}(X)$ is the functional calculus given by $$\theta(f):=\displaystyle\sum_{j=0}^{\infty}f(j)T^j,\quad x\in X,\,f\in \ell^1(\N_0),$$ see \cite[Theorem 5]{Katznelson}. Moreover, for $T\in\mathcal{B}(X)$ power-bounded, $\displaystyle\lim_{n\to\infty}\lVert T^n-T^{n+1}\rVert=0$ if and only if $\sigma(T)\cap \mathbb{T}\subseteq\{1\},$ see \cite[Theorem 1]{Katznelson}.

L\'eka (\cite{L}) proved that for $T$ power-bounded in a Hilbert space $H,$ the result \cite[Theorem 5]{Katznelson} holds if $\mathfrak{f}\in A_+(\mathbb{T})$ just vanishes on the peripheral spectrum. For contractions on $H$ this had been proved in \cite{Es-St_Zo2}. When the Fourier coefficients of $\mathfrak{f}$ satisfy $\sum_{j\geq 0}j|\widehat{\mathfrak{f}}(j)|<\infty,$ the same holds in any Banach space (\cite{AOR}). In the case that $T$ is $(C,1)$-bounded and $\sigma(T)\cap \mathbb{T}=\{1\},$ but $T$ is not power-bounded, $\lVert T^n-T^{n+1}\rVert$ does not need converge to zero. The first counter-examples are in \cite{To-Ze}. There is a counter-example in a Hilbert space with $\sigma(T)=\{1\}$ in \cite{L2}.

A similar result for $C_0$-semigroups was proved simultaneously in two papers, \cite{Es-St_Zo} and \cite{Vu}. The result states that if $(T(t))_{t\geq 0}\subset \mathcal{B}(X)$ is a bounded $C_0$-semigroup generated by $A$ and $\mathfrak{f}\in L^1(\R_+)$ is of spectral synthesis in $L^1(\R)$ with respect to $i\sigma(A)\cap \R,$ then $$\displaystyle\lim_{t\to\infty}\lVert T(t)\Theta(\mathfrak{f}) \rVert=0,$$ where $\Theta:L^1(\R_+)\to\mathcal{B}(X)$ is the Hille functional calculus given by $$\Theta(\mathfrak{f})x:=\displaystyle\int_{0}^{\infty}\mathfrak{f}(t)T(t)x,\quad x\in X,\, \mathfrak{f}\in L^1(\R_+).$$ In the paper \cite[Theorem 5.5]{Ch-To}, there is a nice proof of this result, which has inspired the proof of the main theorem of this paper (Theorem \ref{main}).

In \cite{GMM}, the authors give a similar theorem for $\alpha$-times integrated semigroups: let $\alpha>0,$  $(T_{\alpha}(t))_{t\geq 0}\subset \mathcal{B}(X)$ be an $\alpha$-times integrated semigroup generated by $A$ such that $\sup_{t>0}t^{-\alpha}\lVert T_{\alpha}(t) \rVert$ $<\infty,$ and let $\mathfrak{f}\in \mathcal{T}_+^{(\alpha)}(t^{\alpha})$ be of spectral synthesis in $\mathcal{T}^{(\alpha)}(|t|^{\alpha})$ (both are Sobolev subalgebras of $L^1(\R_+)$ and $L^1(\R)$ respectively which have been studied in detail in \cite{GM}) with respect to $i\sigma(A)\cap \R.$ Then $$\displaystyle\lim_{t\to\infty}t^{-\alpha}\lVert T_{\alpha}(t)\Theta_{\alpha}(\mathfrak{f}) \rVert=0,$$ where $\Theta_{\alpha}:\mathcal{T}_+^{(\alpha)}(t^{\alpha})\to\mathcal{B}(X)$ is the bounded algebra homomorphism defined by $$\Theta_{\alpha}(\mathfrak{f})x:=\displaystyle\int_{0}^{\infty}\mathcal{W}_+^{\alpha}\mathfrak{f}(t)T_{\alpha}(t)x,\quad x\in X,\, \mathfrak{f}\in \mathcal{T}_+^{(\alpha)}(t^{\alpha})$$ and $\mathcal{W}_+^{\alpha}\mathfrak{f}$ is the Weyl fractional derivative of order $\alpha$ of $\mathfrak{f}.$

Let $\alpha>0$ and $T\in\mathcal{B}(X).$ The {\it Ces\`{a}ro sum} of order $\alpha>0$ of $T$ is the family of operators $(\Delta^{-\alpha} \mathcal{T}(n))_{n\in \N_0}\subset \mathcal{B}(X)$ defined by
\begin{equation*}
\Delta^{-\alpha} \mathcal{T}(n)x := \displaystyle\sum_{j=0}^n k^{\alpha}(n-j)T^j x, \qquad x\in X, \quad n\in \N_0,
\end{equation*}
and the {\it Ces\`{a}ro mean} of order $\alpha>0$ of $T$ is the family of operators $(M^{\alpha}_T(n))_{n\in\N_0}$ given by $$M_T^{\alpha}(n)x:=\frac{1}{k^{\alpha+1}(n)}\Delta^{-\alpha} \mathcal{T}(n)x,\qquad x\in X, \quad n\in \N_0,$$
where
\begin{equation*}\label{eq1}
k^{\alpha}(n) := \frac{\Gamma(\alpha +n)}{\Gamma(\alpha) \Gamma(n+1)}={n+\alpha-1\choose \alpha-1}, \qquad n\in  \N_0,
\end{equation*}
is the {\it Ces\`{a}ro kernel} of order $\alpha.$ When the Ces\`{a}ro mean of order $\alpha$ of $T$ is uniformly bounded, that is, \begin{equation*}
\sup_{n} \|M_T^{\alpha}(n) \| < \infty,
\end{equation*}
it is said that the operator $T$ is {\it Ces\`{a}ro bounded} of order $\alpha$ or simply {\it $(C,\alpha)$-bounded}. We extend the Ces\`{a}ro kernel for $\alpha=0$ using that $k^0(n):=\lim_{\alpha\to 0^+}k^{\alpha}(n)=\delta_{n,0}$ for $n\in\N_0,$ where $\delta_{n,j}$ for $n,j\in\Z$ is the known Kronecker delta, i.e., $\delta_{n,j}=1$ if $j=n$ and $0$ otherwise. Then $(C,0)$-boundedness is equivalent to power-boundedness, and for $\alpha=1$ the operator $T$ is said to be Ces\`{a}ro mean bounded (or simply Ces\`{a}ro bounded). From formulas (1.10) and (1.17) in \cite[p. 77]{Zygmund} it can be proved that for $\beta>\alpha\geq 0$ we have $\sup_{n}\lVert M_T^{\beta}(n)\rVert\leq \sup_{n}\lVert M_T^{\alpha}(n)\rVert\leq\sup_n\lVert T^n\rVert$; in particular if $T$ is a power-bounded operator then $T$ is a $(C, \alpha)$ bounded operator for any $\alpha>0$. The converse is not true in general, see \cite[Propositions 4.3 and 4.4]{LSS}. The Assani matrix $$T= \left( \begin{array}{rrr}
-1 & 2  \\
 0 & -1  \\
 \end{array} \right)
$$ is $(C,1)$-bounded but it is not power bounded since $$T^n= \left( \begin{array}{rrr}
(-1)^n & (-1)^{n+1}2n  \\
 0 & (-1)^n  \\
 \end{array} \right),\quad n\in\N_0,$$ see \cite[Section 4.7]{E2} and \cite[Remark 2.3]{Su-Ze13}.

\begin{lemma}\label{radius} Let $\alpha>0.$ If $T$ is $(C,\alpha)$-bounded, then it has spectral radius $r(T)\leq 1.$
\end{lemma}
\begin{proof}
The proof is a straightforward consequence of \cite[Lemma 2.1]{Su-Ze13} since $T$ is $(C,[\alpha]+1)$-bounded.
\end{proof}

The study of mean ergodic theorems for operators which are not power-bounded started with \cite{Hille}. There are many results concerning ergodicity (\cite{De00, Ed04, E2, Su-Ze13, To-Ze, Yo98}) and about the growth (\cite{LSS, Sa}) of the Ces\`aro sums and of the Ces\`aro mean of order $\alpha.$

In a recent paper \cite{Abadias}, it is proved that the algebraic structure of the Ces\`aro sum of order $\alpha$ of a bounded operator is similar to the algebraic structure of an $\alpha$-times integrated semigroups (\cite[Theorem 3.3]{Abadias}). In \cite[Section 2]{Abadias}, we construct certain weighted convolution algebras. For any $\alpha>0,$ if we consider the weight $k^{\alpha+1},$ we denote these algebras by $\tau^{\alpha}(k^{\alpha+1}),$ which are contained in $\ell^1(\N_0).$ We have characterized the $(C,\alpha)$-boundedness by the existence of an algebra homomorphism from $\tau^{\alpha}(k^{\alpha+1})$ into $\mathcal{B}(X)$ (\cite[Corollary 3.7]{Abadias}).

The outline of this paper is as follows: In section 2 we use Weyl fractional differences to construct Banach algebras $\tau^{\alpha}(|n|^{\alpha})$ contained in $\ell^1(\Z)$ (Theorem \ref{algebras}). The techniques used are similar to those in \cite[Section 2]{Abadias}, and we follow the same steps as in the continuous case (\cite{GM}), adapting the proofs. In section 3 we define fractional Wiener algebras of periodic continuous functions $A^{\alpha}_+(\mathbb{T})$ and $A^{\alpha}(\mathbb{T})$ which are isometrically isomorphic via the Fourier transform to $\tau^{\alpha}(k^{\alpha+1})$ and $\tau^{\alpha}(|n|^{\alpha}),$ respectively. These algebras allow us to state the main theorem of this paper (see Theorem \ref{main}): let $\alpha> 0,$ $T\in\mathcal{B}(X)$ be a $(C,\alpha)$-bounded operator and $\mathfrak{f}\in A_+^{\alpha}(\mathbb{T})$ be of spectral synthesis in $A^{\alpha}(\mathbb{T})$ with respect to $\sigma(T)\cap \mathbb{T}.$ Then $$\displaystyle\lim_{n\to\infty}\lVert M_T^{\alpha}(n)\theta_{\alpha}(\widehat{\mathfrak{f}}) \rVert=0,$$ where $\theta_{\alpha}:\tau^{\alpha}(k^{\alpha+1})\to\mathcal{B}(X)$ is the bounded algebra homomorphism defined by $$\theta_{\alpha}(f)x:=\displaystyle\sum_{n=0}^{\infty}W_+^{\alpha}f(n)\Delta^{-\alpha} \mathcal{T}(n)x,\quad x\in X,\, f\in\tau^{\alpha}(k^{\alpha+1}),$$ and $W_+^{\alpha}f$ is the Weyl fractional difference of order $\alpha$ of $f,$ see \cite[Theorem 3.5]{Abadias}. Finally in section 4 we give two applications of ergodicity for $(C,\alpha)$-bounded operators (Theorem \ref{ergodicity1} and Theorem \ref{ergodicity}).\\

\noindent{\bf Notation.} We denote by $\ell^1(\Z)$ the set of complex sequences \mbox{$f:\Z\to \C$} such that \mbox{$\sum_{n=-\infty}^{\infty}| f(n)|<\infty,$} and $c_{0,0}(\Z)$ the set of complex sequences with finite support. It is well known that $\ell^1(\Z)$ is a Banach algebra with the usual (commutative and associative) convolution product $$(f*g)(n)=\displaystyle\sum_{j=-\infty}^{\infty} f(n-j)g(j),\quad n\in\Z.$$ The above is valid for sequences defined in $\N_0$ instead $\Z,$ and the corresponding convolution product is $$(f*g)(n)=\displaystyle\sum_{j=0}^{n} f(n-j)g(j),\quad n\in\N_0.$$ Moreover, if $f$ is a sequence defined in $\N_0,$ we can see it as a sequence defined in $\Z$ where $f(n)=0$ for $n<0.$

Throughout the paper, we use the variable constant convention, in which $C$ denotes a constant
which may not be the same from line to line. The constant is frequently written with subindexes
to emphasize that it depends on some parameters.

\section{Fractional differences and convolution Banach algebras}
\setcounter{theorem}{0}
\setcounter{equation}{0}

For $\alpha>0,$ the Ces\`aro kernel of order $\alpha,$ $(k^{\alpha}(n))_{n\in\N_0,}$  plays a key role in the main results of this paper. Many properties can be found in \cite[Vol. I, p.77]{Zygmund}.  We quote some of them below: the semigroup property, $k^{\alpha}*k^{\beta}=k^{\alpha+\beta}$ for $\alpha, \beta >0$; for $\alpha>0$,   \begin{equation}\label{double}
 k^{\alpha}(n)=\frac{n^{\alpha-1}}{\Gamma(\alpha)}(1+O({1\over n})), \qquad n\in \N, \end{equation}
(\cite[Vol. I, (1.18)]{Zygmund}); $k^\alpha$ is increasing (as a function of $n$) for $\alpha >1$, decreasing for $0<\alpha<1$ and $k^1(n)=1$ for $n\in \N$ (\cite[Chapter III, Theorem 1.17]{Zygmund}); $k^\alpha(n)\le k^\beta(n)$ for $\beta \ge \alpha>0$ and $n\in \N_0$; finally, for $\alpha>0,$ there exists $C_{\alpha}>0$ such that the following inequality holds,
\begin{equation}\label{eq2.5}
k^\alpha(2n)\le C_{\alpha} k^\alpha(n), \qquad n\in \N_0,
\end{equation}
(\cite[Lemma 2.1]{Abadias}).

As we mentioned in the introduction, for each number $\alpha>0$ there exists a convolution Banach algebra $\tau^{\alpha}(k^{\alpha+1}),$ which is contained in $\ell^1(\N_0)$ and they are continuously included in each other, that is, $$\tau^{\beta}(k^{\beta+1})\hookrightarrow\tau^{\alpha}(k^{\alpha+1})\hookrightarrow\ell^1(\N_0),\quad \beta>\alpha>0,$$ and $\tau^{0}(k^{1})\equiv \ell^1(\N_0),$ see \cite{Abadias}. Now we are interested in obtaining some similar spaces contained in $\ell^1(\Z).$ For convenience, we denote $\tau^{\alpha}(n^{\alpha}):=\tau^{\alpha}(k^{\alpha+1})$ for $\alpha>0.$

In the following, let $(f(n))_{n\in\Z}$ be a sequence of complex numbers. Some results in this section can be extended immediately to vector-valued sequences, that is, $f$ takes values in a complex Banach space $X.$ We consider the usual forward and backward difference operator, $\Delta f(n)=f(n+1)-f(n)$ and $\nabla f(n)=f(n)-f(n-1),$ for $n\in\Z,$ and the natural powers $$\Delta^m f(n)=\displaystyle\sum_{j=0}^m(-1)^{m-j}\binom{m}{j}f(n+j),\qquad n\in\Z,$$ and $$\nabla^m f(n)=\displaystyle\sum_{j=0}^m(-1)^{j}\binom{m}{j}f(n-j),\qquad n\in\Z,$$ for $m\in\N_0,$ see for example \cite[(2.1.1)]{Elaydi} for $\Delta^{m}$ (for $\nabla^m$ it is a simple check using $\Delta^m$). Observe that $\Delta^m,\nabla^m:c_{0,0}(\Z)\to c_{0,0}(\Z)$ for $m\in \N_0$.

For convenience and follow the same notation as in \cite{Abadias}, we write $W_+=-\Delta$ and $W_{-}=\nabla,$ $W_+^{m}=(-1)^{m}\Delta^m$ and $W_-^{m}=\nabla^m$ for $m\in \N$. The inverse operators of $W_+$ and $W_-$, and their powers in $c_{0,0}(\Z)$ are given by the following expressions,
$$W_+^{-m}f(n)=\sum_{j=n}^{\infty}k^{m}(j-n)f(j),\qquad n\in \Z,$$ and $$W_-^{-m}f(n)=\sum_{j=-\infty}^{n}k^{m}(n-j)f(j),\qquad n\in \Z$$ for $m\in\N,$ see for example \cite[p.307]{Gale} in the case of $W_+$ for sequences define in $\N_0.$

\begin{definition}\label{WeylDifference}{\rm Let $(f(n))_{n\in\Z}$ be a complex sequence and $\alpha>0.$ The  {\it Weyl sums} of order $\alpha$ of $f$ are given by $$W_{+}^{-\alpha}f(n):=\displaystyle\sum_{j=n}^{\infty} k^{\alpha}(j-n)f(j),\qquad n\in\Z,$$ and $$W_{-}^{-\alpha}f(n):=\displaystyle\sum_{j=-\infty}^{n} k^{\alpha}(n-j)f(j),\qquad n\in\Z,$$
whenever the sums make sense, and the {\it Weyl differences} by $$W_+^{\alpha}f(n):=W_+^m W_+^{-(m-\alpha)}f(n)=(-1)^{m}\Delta^m W_+^{-(m-\alpha)}f(n),\qquad n\in\Z,$$ and $$W_-^{\alpha}f(n):=W_-^m W_-^{-(m-\alpha)}f(n)=\nabla^m W_-^{-(m-\alpha)}f(n),\qquad n\in\Z,$$ for $m=[\alpha]+1,$  whenever the right hand sides converge. In particular $W_+^{\alpha},\,W_-^{\alpha}: c_{0,0}(\Z)\to c_{0,0}(\Z)$ for $\alpha \in \R$. }
\end{definition}

The above definitions have been considered in more restrictive contexts in some papers (\cite{Abadias,Gale}). The natural properties that are satisfied in those contexts are generalized below, and the proof is similar to the proof of \cite[Proposition 2.4]{Abadias}.

\begin{proposition}\label{WeylSumProp} Let $f\in c_{0,0}(\Z)$ and $\alpha,\beta\in\R,$ then the following statements hold:\begin{itemize}
\item[(i)] $W_+^{\alpha+\beta}f=W_+^{\alpha}W_+^{\beta}f.$
\item[(ii)] $W_-^{\alpha+\beta}f=W_-^{\alpha}W_-^{\beta}f.$
\item[(iii)] $\displaystyle\lim_{\alpha\to 0}W_+^{\alpha}f=\displaystyle\lim_{\alpha\to 0}W_-^{\alpha}f=f.$
\end{itemize}
\end{proposition}

Note that the Ces\`aro kernel can be considered in a more general setting. For $\alpha\in\R,$ $$k^{\alpha}(n)=\frac{\alpha(\alpha+1)\cdots (\alpha+n-1)}{n!},\ \text{ for }n\in\N,\ k^{\alpha}(0)=1.$$ Also in the particular case that $\alpha<0$ with $\alpha\neq\{0,-1,-2,\ldots\}$ we can write $k^{\alpha}(n)=(-1)^n\binom{-\alpha}{n}$ and \eqref{double} is valid too. It is known that $$\displaystyle\sum_{n=0}^{\infty}k^{\alpha}(n)z^n=(1-z)^{-\alpha},\quad |z|<1.$$ Then we deduce that $k^{\alpha}*k^{\beta}=k^{\beta+\alpha}$ for $\alpha,\beta\in\R.$ This allows to represent the Weyl differences in the following way.

\begin{proposition}\label{other} Let $(f(n))_{n\in\Z}$ be a complex sequence and $\alpha \in\R.$ Then $$W_+^{\alpha}f(n)=\displaystyle\sum_{j=n}^{\infty} k^{-\alpha}(j-n)f(j),\quad W_-^{\alpha}f(n)=\displaystyle\sum_{j=-\infty}^{n} k^{-\alpha}(n-j)f(j),\qquad n\in\Z,$$ whenever the Weyl differences of order $\alpha$ of $f$ makes sense.
\end{proposition}
\begin{proof}
We only prove the result for $W_+.$ The proof is analogous for $W_-.$ If $\alpha\in\N,$ then $$W_+^{\alpha}f(n)=\displaystyle\sum_{j=0}^{\alpha}(-1)^j\binom{\alpha}{j}f(n+j)=\displaystyle\sum_{j=0}^{\infty}k^{-\alpha}(j)f(n+j)=\displaystyle\sum_{j=n}^{\infty}k^{-\alpha}(n-j)f(j).$$ Now let $m-1<\alpha<m$ with $m\in\N.$ Then \begin{eqnarray*}
W^{\alpha}_+f(n)&=&W_+^mW_+^{-(m-\alpha)}f(n)=\displaystyle\sum_{j=0}^{m}(-1)^j\binom{m}{j}\displaystyle\sum_{l=n+j}^{\infty}k^{m-\alpha}(l-n-j)f(l)\\
&=&\displaystyle\sum_{l=n}^{n+m}f(l)\displaystyle\sum_{j=0}^{l-n}(-1)^j\binom{m}{j}k^{m-\alpha}(l-n-j)\\
&&+\displaystyle\sum_{l=n+m+1}^{\infty}f(l)\displaystyle\sum_{j=0}^{m}(-1)^j\binom{m}{j}k^{m-\alpha}(l-n-j)\\
&=&\displaystyle\sum_{l=n}^{n+m}f(l)(k^{-m}*k^{m-\alpha})(l-n)\\
&&+\displaystyle\sum_{l=n+m+1}^{\infty}f(l)(k^{-m}*k^{m-\alpha})(l-n)\\
&=&\displaystyle\sum_{l=n}^{\infty}k^{-\alpha}(l-n)f(l).
\end{eqnarray*}
\end{proof}

\begin{remark}{\rm The operators $W_+^{\alpha}$ and $W_+^{-\alpha}$ for $\alpha\in (0,1)$ are tightly connected to the definition of $(I-T)^{\alpha}$ for any contraction $T$ in a Banach space given in \cite{DL}. If we denote by $S$ the shift operator on $\ell^1(\Z)$, that is, $Sf(n)=f(n+1)$ for $n\in\Z,$ then $W_+^{\alpha}=(I-S)^{\alpha}$ (compatible with $I-S=-\Delta$), well-defined on the whole space $\ell^1(\Z),$ and $W_+^{-\alpha}=[(I-S)^{\alpha}]^{-1},$ defined on the range of $(I-S)^{\alpha}.$ Also these identities are valid for $\alpha>0.$ The author is studying these fractional powers of the operator $I-S$ as fractional powers of the generator of a uniformly bounded $C_0$-semigroup on a Banach space, which will appear in a forthcoming paper.
}
\end{remark}

\begin{remark}{\rm  Note that $W_+^mf(n)=\displaystyle\sum_{j=0}^m(-1)^{j}\binom{m}{j}f(n+j)$ and $W_-^mf(n)=\displaystyle\sum_{j=0}^m(-1)^{j}\binom{m}{j}f(n-j)$ for $m\in\N$ and $n\in\Z,$ therefore in general $W_+^{\alpha}f(n)\neq W_-^{\alpha}f(n)$ for $\alpha>0$ and $n\in\Z$ (it suffices take $0<\alpha<1$ and the sequence given by $f(n)=1$ for $n=0,1,$ and $f(n)=0$ in otherwise). However we have the following link between $W_+^{\alpha}$ and $W_-^{\alpha}.$ The proof is left to the reader.
}
\end{remark}

\begin{proposition} Let $\alpha$ be a positive real number and $f\in c_{0,0}(\Z)$ such that $f(n)=f(-n)$ for all $n\in\Z.$ Then the equality $$W_+^{\alpha}f(n)=W_-^{\alpha}f(-n),\quad n\in\Z,$$ holds. In particular $W_+^{\alpha}f(0)=W_-^{\alpha}f(0).$
\end{proposition}

Let $(f(n))_{n\in\Z}$ be a complex sequence, we denote by $(f_+(n))_{n\in\Z},$ $(f_-(n))_{n\in\Z}$ and $(\tilde{f}(n))_{n\in\Z}$ the sequences given by \begin{displaymath}
f_+(n):=\left\{\begin{array}{ll}
f(n),& n\geq 0, \\
0,& n<0 ,
\end{array} \right.
\end{displaymath}

\begin{displaymath}
f_-(n):=\left\{\begin{array}{ll}
0,& n\geq 0, \\
f(n),& n<0,
\end{array} \right.
\end{displaymath}
and $\tilde{f}(n)=f(-n)$ for $n\in\Z.$ It is a simple check that $(W_+^{-\alpha}f)\tilde(n)=W_-^{-\alpha}\tilde{f}(n),$  $n\in\Z,$ for $\alpha>0$ and $f\in c_{00}(\Z).$ Then the following result is a straight consequence.
\begin{proposition}\label{prop2.5} Let $f\in c_{0,0}(\Z)$ and $\alpha>0,$ then the following assertions hold:\begin{itemize}
\item[(i)] $W_+^{\alpha}f_+(n)=W_+^{\alpha}f(n),\quad n\geq 0.$
\item[(ii)] $W_-^{\alpha}f_-(n)=W_-^{\alpha}f(n),\quad n<0.$
\item[(iii)] $(W_+^{\alpha}f)\tilde(n)=W_-^{\alpha}\tilde{f}(n),\quad n\in\Z.$
\end{itemize}
\end{proposition}

\begin{definition}{\rm Let $\alpha>0.$ We denote by $W^{\alpha}:c_{0,0}(\Z)\to c_{0,0}(\Z)$ the operator given by \begin{displaymath}
W^{\alpha}f(n):=\left\{\begin{array}{ll}
W_+^{\alpha}f(n),& n\geq 0, \\ \\
W_-^{\alpha}f(n),& n<0 ,
\end{array} \right.
\end{displaymath}
for $f\in c_{0,0}(\Z).$}
\end{definition}

We are interested in the relation between the convolution product and the fractional Weyl differences. If $f,g\in c_{0,0}(\Z)$ then it is known that $f*g\in c_{0,0}(\Z).$ In \cite[Lemma 2.7]{Abadias}, the following equality is proved: \begin{equation}\label{eq2.7}\begin{array}{rcl}
W_+^{\alpha}(f_+*g_+)(n)&=&\displaystyle\sum_{j=0}^n W_+^{\alpha}g(j)\displaystyle\sum_{p=n-j}^n k^{\alpha}(p-n+j)W_+^{\alpha}f(p) \\
&&-\displaystyle\sum_{j=n+1}^{\infty} W_+^{\alpha}g(j)\displaystyle\sum_{p=n+1}^{\infty} k^{\alpha}(p-n+j)W_+^{\alpha}f(p),\quad n\geq 0,
\end{array}
\end{equation}
for $f,g\in c_{0,0}(\Z)$ and $\alpha\geq 0.$ The rest of this section is inspired by the continuous case, see \cite{GM}.

\begin{lemma}\label{lemma2.8} Let $f, g\in c_{0,0}(\Z)$ and $\alpha>0,$ then \begin{itemize}
\item[(i)] $W_+^{\alpha}(f_+*g_-)(n)=(W_+^{\alpha}f_+*g_-)(n),\quad n\geq 0.$
\item[(ii)] $W_-^{\alpha}(f_-*g_+)(n)=(W_-^{\alpha}f_-*g_+)(n),\quad n< 0.$
\end{itemize}
\end{lemma}
\begin{proof} (i) Let $n\geq 0,$ then \begin{eqnarray*}
(f_+*g_-)(n)&=&\displaystyle\sum_{j=n+1}^{\infty}W_{+}^{-\alpha}W_{+}^{\alpha}f_+(j)g_-(n-j)\\
&=&\displaystyle\sum_{j=n+1}^{\infty}W^{\alpha}_+f_+(j)\sum_{i=n+1}^jk^{\alpha}(j-i)g_-(n-i) \\&=&\displaystyle\sum_{j=n+1}^{\infty}W^{\alpha}_+f_+(j)\sum_{u=n}^{j-1}k^{\alpha}(u-n)g_-(u-j)\\
&=&\sum_{u=n}^{\infty}k^{\alpha}(u-n)\displaystyle\sum_{j=u+1}^{\infty}W^{\alpha}_+f_+(j)g_-(u-j)\\
&=& W_+^{-\alpha}(W_+^{\alpha}f_+*g_-)(n),
\end{eqnarray*} where we have used Fubini's Theorem and a change of variables, and then $W_+^{\alpha}(f_+*g_-)(n)=W_+^{\alpha}f_+*g_-(n).$ (ii) Using Proposition \ref{prop2.5} and the part (i) we get for $n<0$ that \begin{eqnarray*}
W_-^{\alpha}(f_-*g_+)(n)&=&W_+^{\alpha}(f_-*g_+)\tilde\,(-n)=W_+^{\alpha}((f_-)\tilde\,*(g_+)\tilde\,)(-n)\\ \\
&=&W_+^{\alpha}(\tilde{f}_+*\tilde{g}_-)(-n)=(W_+^{\alpha}\tilde{f}_+*\tilde{g}_-)(-n)\\ \\
&=&( (W_+^{\alpha}\tilde{f}_+)\tilde\,*(\tilde{g}_-)\tilde\, )(n)=(W_-^{\alpha}f_-*g_+)(n).
\end{eqnarray*}
\end{proof}

%\begin{lemma}\label{lemma2.9} Let $f, g\in c_{00}(\Z)$ and $\alpha>0,$ then \begin{displaymath}
%W^{\alpha}(f*g)(n)=\left\{\begin{array}{ll}
%(W_+^{\alpha}f_+*g_-)(n) + W_+^{\alpha}(f_+*g_+)(n) + (f_-*W_+^{\alpha}g_+)(n),& n\geq 0, \\ \\
%(W_-^{\alpha}f_-*g_+)(n) + W_-^{\alpha}(f_-*g_-)(n)+  (f_+*W_-^{\alpha}g_-)(n),& n<0 .
%\end{array} \right.
%\end{displaymath}
%\end{lemma}

\begin{lemma}\label{lemma2.9} Let $f, g\in c_{0,0}(\Z)$ and $\alpha>0,$ then $$
W^{\alpha}(f*g)(n)=(W_+^{\alpha}f_+*g_-)(n) + W_+^{\alpha}(f_+*g_+)(n) + (f_-*W_+^{\alpha}g_+)(n),$$ for $n\geq 0,$ and $$
W^{\alpha}(f*g)(n)=(W_-^{\alpha}f_-*g_+)(n) + W_-^{\alpha}(f_-*g_-)(n)+  (f_+*W_-^{\alpha}g_-)(n),$$ for $n<0$.
\end{lemma}
\begin{proof}
It is a simple check that $$(f*g)(n)=(f_+*g_-)(n) + (f_+*g_+)(n) + (f_-*g_+)(n),\quad n\geq 0$$ and $$(f*g)(n)=(f_-*g_+)(n) + (f_-*g_-)(n)+ (f_+*g_-)(n),\quad n<0.$$ Then by Lemma \ref{lemma2.8} we get the result.
\end{proof}

For $\alpha\geq 0$ we define the application $q_{\alpha}:c_{0,0}(\Z)\to [0,\infty)$ given by $$q_{\alpha}(f):=\displaystyle\sum_{n=-\infty}^{\infty}k^{\alpha+1}(|n|)|W^{\alpha}f(n)|, \qquad f\in c_{0,0}(\Z).$$ Observe that for $\alpha=0$ the above application is the usual norm in $\ell^1(\Z)$.

The following theorem is the main one of this section, and it extends \cite[Theorem 2.11]{Abadias} and \cite[Theorem 4.5]{Gale}.

\begin{theorem}\label{algebras}\label{th3.1} Let $\alpha>0.$ The application $q_{\alpha}$ defines a norm in $c_{0,0}(\Z)$ and  $$q_{\alpha}(f*g)\leq C_{\alpha}\,q_{\alpha}(f)\,q_{\alpha}(g), \qquad f,g\in c_{0,0}(\Z),$$ with $C_{\alpha}>0$ independent of $f$ and $g.$ We denote by $\tau^{\alpha}(|n|^{\alpha})$ the Banach algebra obtained as the space of complex sequences $f$ such that $\lim_{n\to\infty} f(n)=0$ and the norm $q_{\alpha}(f)$ converges. Furthermore, these spaces are continuously embedding each in other in the following way
$$\tau^{\beta}(|n|^{\beta})\hookrightarrow\tau^{\alpha}(|n|^{\alpha})\hookrightarrow \ell^1(\Z),$$ for $\beta>\alpha>0,$ and $\lim_{\alpha \to 0^+}q_\alpha(f)=\Vert f\Vert_1$, for $f\in c_{0,0}(\Z)$.
\end{theorem}
\begin{proof} It is clear that $q_{\alpha}$ is a norm in $c_{0,0}(\Z).$ We write \begin{eqnarray*}
q_{\alpha}(f)&=&\displaystyle\sum_{n=-\infty}^{-1}k^{\alpha+1}(-n)|W^{\alpha}_-f_-(n)|+\displaystyle\sum_{n=0}^{\infty}k^{\alpha+1}(n)|W^{\alpha}_+f_+(n)|\\
&:=&q_{\alpha}^-(f_-)+q_{\alpha}^+(f_+).\end{eqnarray*}
We have to show that $q_{\alpha}$ defines a Banach algebra. First we prove that $$q_{\alpha}^{+}((f*g)_+)\leq C_{\alpha}q_{\alpha}(f)q_{\alpha}(g).$$ By Lemma \ref{lemma2.9}, $$W^{\alpha}(f*g)(n)=(W_+^{\alpha}f_+*g_-)(n)+W_+^{\alpha}(f_+*g_+)(n)+(f_-*W_+^{\alpha}g_+)(n),$$ for $n\geq 0,$ then we work with each summand separately. The first, \begin{eqnarray*}
\displaystyle\sum_{n=0}^{\infty}k^{\alpha+1}(n)|(W_+^{\alpha}f_+*g_-)(n)|&\leq &\sum_{n=0}^{\infty}k^{\alpha+1}(n)\sum_{j=n+1}^{\infty}|W_+^{\alpha}f_+(j)||g_-(n-j)| \\
&=&\sum_{j=1}^{\infty}|W_+^{\alpha}f_+(j)|\sum_{n=0}^{j-1}k^{\alpha+1}(n)|g_-(n-j)| \\
&\leq&\sum_{j=1}^{\infty}|W_+^{\alpha}f_+(j)|k^{\alpha+1}(j)\sum_{u=-j}^{-1}|g_-(u)| \\
&\leq&q_{\alpha}^+(f_+)q_{\alpha}^-(g_-)\leq q_{\alpha}(f)q_{\alpha}(g),
\end{eqnarray*}
where we have used Fubini's Theorem, a change of variables and that $k^{\alpha+1}$ is increasing (as function of n) for $\alpha>0.$ The third is clear using the commutativity of the convolution and the bound of the first summand. The second is a consequence of Proposition \ref{prop2.5} (i) and \cite[Theorem 2.11]{Abadias}.

To finish we have to estimate $q_{\alpha}^{-}((f*g)_-).$ By Proposition \ref{prop2.5} (ii) we have for $n<0$ that $$W_-^{\alpha}(f*g)(n)=W_+^{\alpha}(f*g)\tilde\,(-n)=W_+^{\alpha}(\tilde{f}*\tilde{g})(-n)=W_+^{\alpha}((\tilde{f}*\tilde{g})_+)(-n),$$ then $$q_{\alpha}^-((f*g)_-)\leq \displaystyle\sum_{n=0}^{\infty}k^{\alpha+1}(n)|W_+^{\alpha}(\tilde{f}*\tilde{g})_+(n)|\leq C_{\alpha}q_{\alpha}(\tilde{f})q_{\alpha}(\tilde{g})=C_{\alpha}q_{\alpha}(f)q_{\alpha}(g).$$

Finally note that if $f\in \tau^{\beta}(|n|^{\beta}),$ then
 \begin{eqnarray*}
q_{\alpha}(f)&=&\displaystyle\sum_{n=-\infty}^{\infty}|W^{\alpha}f(n)|k^{\alpha+1}(n)=\displaystyle\sum_{n=0}^{\infty}k^{\alpha+1}(n)|\displaystyle\sum_{j=n}^{\infty}k^{\beta-\alpha}(j-n)W_+^{\beta}f(j)| \\
&&+\displaystyle\sum_{n=-\infty}^{-1}k^{\alpha+1}(-n)|\displaystyle\sum_{j=-\infty }^{n}k^{\beta-\alpha}(n-j)W_-^{\beta}f(j)|\\
&\leq&\displaystyle\sum_{j=0}^{\infty}|W_+^{\beta}f(j)|k^{\beta+1}(j)+\displaystyle\sum_{j=-\infty }^{-1}|W_-^{\beta}f(j)|k^{\beta+1}(-j)\\
&=&\displaystyle\sum_{j=-\infty}^{\infty}k^{\beta+1}(|j|)|W^{\beta}f(j)|=q_{\beta}(f),
\end{eqnarray*} where we have applied Proposition \ref{WeylSumProp} and the semigroup property of $k^{\alpha}.$
\end{proof}

\begin{remark} {\rm Note that by \eqref{double} the norm $q_{\alpha}$ is equivalent to the norm $\overline{q_{\alpha}}$ where \begin{eqnarray*}
\overline{q_{\alpha}}(f)&:=&\displaystyle\sum_{n=1}^{\infty}n^{\alpha}|W^{\alpha}_-f(-n)|+|f(0)|+\displaystyle\sum_{n=1}^{\infty}n^{\alpha}|W^{\alpha}_+f(n)| \\
&=&|f(0)|+\displaystyle\sum_{n=1}^{\infty}n^{\alpha}(|W^{\alpha}_+f(n)|+|W^{\alpha}_+\tilde{f}(n)|).
\end{eqnarray*}}

\end{remark}

\section{A Katznelson-Tzafriri type theorem for $(C,\alpha)$-bounded operators}

\setcounter{theorem}{0}
\setcounter{equation}{0}

For $\alpha> 0,$ we denote by $A^{\alpha}(\mathbb{T})$ a new Wiener algebra formed by all continuous periodic functions $\mathfrak{f}(t)=\sum_{n=-\infty}^{\infty}\widehat{\mathfrak{f}}(n) e^{int},$ for $t\in [0,2\pi],$ with the norm $$\lVert \mathfrak{f} \rVert_{A^{\alpha}(\mathbb{T})}:=\displaystyle\sum_{n=-\infty}^{\infty}|W^{\alpha}\widehat{\mathfrak{f}}(n)| k^{\alpha+1}(|n|)<\infty.$$ This algebra is regular since its character is equal to the character of $\ell^1(\Z)$, which is $\mathbb{T}.$ Similarly to the case $\alpha=0,$ we denote by $A_+^{\alpha}(\mathbb{T})$ the convolution closed subalgebra of $A^{\alpha}(\mathbb{T})$ where the coefficients $\widehat{\mathfrak{f}}(n)=0$ for $n<0.$ Note that $A^{\alpha}(\mathbb{T})$ and $\tau^{\alpha}(|n|^{\alpha})$ are isometrically isomorphic via Fourier coefficients. The same holds for $A_+^{\alpha}(\mathbb{T})$ and $\tau^{\alpha}(n^{\alpha}).$

\begin{remark}{\rm
Recall that in \cite{AOR} a version of the original Katznelson-Tzafriri Theorem is proved for periodic functions $\mathfrak{f}$ such that their Fourier coefficients satisfy $\sum_{j\geq 0}j|\widehat{\mathfrak{f}}(j)|<\infty.$ Note that $A_+^1(\T)$ includes these functions. Also, if $\mathfrak{f}\in A_+(\T)$ has monotonely decreasing Fourier coefficients then, by remark at the beginning of the proof of \cite[Theorem 4.1, Chapter 1]{Kat}  for the sequence $a_n=\sum_{j\geq n}\widehat{\mathfrak{f}}(j),$ $\mathfrak{f}\in A_+^1(\T).$ This class of functions is studied in \cite{Zygmund}.

More generally, the subalgebras $A_+^m(\T)$ for $m\in\N$ are larger than the Korenblyum subalgebras defined in \cite{Gale}. In fact, $\mathfrak{f}(t)=\sum_{n\geq 1}\frac{1}{n^{m+1}}e^{int}\in A_+^m(\T)$ and $\mathfrak{f}$ does not belong to the corresponding Korenblyum subalgebra.
}
\end{remark}

\begin{remark}{\rm
The spaces $A_+^{\alpha}(\T)$ are decreasing as $\alpha$ increases, and they are dense in $A_+(\T),$ since by \cite{Gale} those with integer $\alpha$ are. Furthermore we have the following:\begin{itemize}

\item[(i)] Note that the proof of \cite[Theorem 2.10 (iii)]{Abadias} proves that $\lVert \mathfrak{f}\rVert_{A_+^{\alpha}(\T)}\leq \lVert \mathfrak{f}\rVert_{A_+^{\beta}(\T)}$ for $0\leq \alpha <\beta,$ then these spaces are continuously embedding each in other with norm 1.

\item[(ii)] For $0\leq \alpha<\beta$ we have $A_+^{\beta}(\T)\subsetneq A_+^{\alpha}(\T).$ This is a consequence of the characterization of the $(C,\alpha)$-boundedness by means of homomorphisms defined on these spaces (\cite[Corollary 3.7]{Abadias}) and the existence of operators which are $(C,\beta)$-bounded but no $(C,\alpha)$-bounded (\cite[Propositions 4.3 and 4.4]{LSS}).

\item[(iii)] For $\alpha>0,$ the functions $\mathfrak{f}$ such that $\sum_{j\geq 0}j^{\alpha}|\widehat{\mathfrak{f}}(j)|<\infty$ are included in $A_+^{\alpha}(\T).$ In fact, note that there exits a sequence $c\in\ell^{\infty}(\N_0),$ such that $|\mathfrak{\widehat{f}}(n+j)|<c(n)|\mathfrak{\widehat{f}}(n)|$ for all $j\geq 0.$ By Proposition \ref{other} and \eqref{double} for $k^{-\alpha}$ we have $$\overline{q_{\alpha}}({\mathfrak{\widehat{f}}})\leq |\mathfrak{\widehat{f}}(0)|+\lVert c\rVert_{\infty}\sum_{j=0}^{\infty}|k^{-\alpha}(j)|\sum_{n=0}^{\infty}n^{\alpha}|\mathfrak{\widehat{f}}(n)|<\infty.$$
\end{itemize}
}
\end{remark}

Let $E$ be a closed subset of $\mathbb{T}$ and $ \mathfrak{f}\in A^{\alpha}(\mathbb{T}).$ We recall that $ \mathfrak{f}$ is of spectral synthesis with respect to $E$ if for every $\varepsilon>0$ there exists $ \mathfrak{f}_{\varepsilon}\in A^{\alpha}(\mathbb{T})$ such that $\lVert  \mathfrak{f}- \mathfrak{f}_{\varepsilon} \rVert_{A^{\alpha}(\mathbb{T})}<\varepsilon$ with $ \mathfrak{f}_{\varepsilon}= 0$ in a neighborhood of $E.$

Let $T\in \mathcal{B}(X)$ and $\alpha>0.$ We can write the $(C,\alpha)$-boundedness of $T$ in the following way: there exists a constant $C>0$ such that \begin{equation*}
\|\Delta^{-\alpha} \mathcal{T}(n) \| \leq C k^{\alpha+1}(n),\quad n\in\N_0.
\end{equation*}
Furthermore, we have cited in the introduction that for $\alpha> 0$ and $T\in\mathcal{B}(X)$ a $(C,\alpha)$-bounded operator, there exists a bounded algebra homomorphism $\theta_{\alpha}:\tau^{\alpha}(n^{\alpha})\to \mathcal{B}(X)$ given by $$\theta_{\alpha}(f)x=\displaystyle\sum_{n=0}^{\infty}W^{\alpha}_+f(n)\Delta^{-\alpha} \mathcal{T}(n)x,\qquad x\in X,\,f\in \tau^{\alpha}(n^{\alpha}),$$ see \cite[Theorem 3.5]{Abadias}.

\begin{remark}{\rm
Let $T\in\mathcal{B}(X)$ and $\alpha>0.$ We can write $T^j=(k^{-\alpha}*\Delta^{-\alpha} \mathcal{T})(j).$ \begin{itemize}
\item[(i)]Let $m$ be a positive integer and $T$ a $(C,m)$-bounded operator. Since $k^{-m}\in c_{0,0}(\N_0),$ then $\lVert T^j\rVert=O(j^{m}).$ So, if $\mathfrak{f}$ belongs to the Korenblyum subalgebra ($\sum_{j\geq 0}j^m|\widehat{\mathfrak{f}}(j)|<\infty$) we have that $\sum_{j\geq 0}\widehat{\mathfrak{f}}(j)T^j$ converges in operator norm. Moreover, by induction method, \begin{equation*}\begin{array}{l}
    \displaystyle \theta_m(\widehat{\mathfrak{f}})=\lim_{N\to\infty}\sum_{n=0}^{N}W_+^m\widehat{\mathfrak{f}}(n)\Delta^{-m}\mathcal{T}(n)=\lim_{N\to\infty}\biggl(\sum_{j=0}^N\widehat{\mathfrak{f}}(j)T^j\\
    \displaystyle+(-1)^{m+1}\sum_{j=0}^{m-1}W^j_+\widehat{\mathfrak{f}}(N+1)\Delta^{-(j+1)}\mathcal{T}(N)\biggr)=\lim_{N\to\infty}\biggl(\sum_{j=0}^N\widehat{\mathfrak{f}}(j)T^j\\
    \displaystyle+(-1)^{m+1}\sum_{j=0}^{m-1}\left(\sum_{l=0}^{j}(-1)^l\binom{j}{l}\widehat{\mathfrak{f}}(N+1+l)\right)\Delta^{-(j+1)}\mathcal{T}(N)\biggr)\\
    \displaystyle=\sum_{j=0}^\infty\widehat{\mathfrak{f}}(j)T^j.
    \end{array}\end{equation*}

\item[(ii)]Let $\alpha>0$ be a non positive integer and $T$ a $(C,\alpha)$-bounded operator. First observe that the sign of $k^{-\alpha}(j)$ is $(-1)^{[\alpha]+1}$ for all $j\geq[\alpha]+1.$ For $j\geq[\alpha]+1,$ note that \begin{equation*}\begin{array}{l}\lVert T^j\rVert\leq C \displaystyle\sum_{n=0}^j|k^{-\alpha}(j-n)|k^{\alpha+1}(n)\\
    =C\left(  (-1)^{[\alpha]+1}\displaystyle\sum_{n=0}^{j-[\alpha]-1}k^{-\alpha}(j-n)k^{\alpha+1}(n)+\displaystyle\sum_{n=j-[\alpha]}^{j}|k^{-\alpha}(j-n)|k^{\alpha+1}(n)\right)\\
    =C\biggl(  (-1)^{[\alpha]+1}\displaystyle\sum_{n=0}^{j}k^{-\alpha}(j-n)k^{\alpha+1}(n)\\
    +\displaystyle\sum_{n=j-[\alpha]}^{j}(|k^{-\alpha}(j-n)|-(-1)^{[\alpha]+1}k^{-\alpha}(j-n))k^{\alpha+1}(n)\biggr)\\
    \leq C_{\alpha}\biggl( (-1)^{[\alpha]+1}+k^{\alpha+1}(j) \biggr),
    \end{array}\end{equation*}
    where we have used that $k^{\alpha+1}$ is increasing and $k^{-\alpha}*k^{\alpha+1}=k^1.$ Then $\lVert T^j\rVert=O(j^{\alpha}).$ So, if $\mathfrak{f}$ belongs to the extended Korenblyum subalgebra ($\sum_{j\geq 0}j^\alpha|\widehat{\mathfrak{f}}(j)|<\infty$) we have that $\sum_{j\geq 0}\widehat{\mathfrak{f}}(j)T^j$ converges in operator norm. Moreover, \begin{equation*}\begin{array}{l}
    \displaystyle \theta_\alpha(\widehat{\mathfrak{f}})=\lim_{N\to\infty}\sum_{n=0}^{N}W_+^\alpha\widehat{\mathfrak{f}}(n)\Delta^{-\alpha}\mathcal{T}(n)\\
    \displaystyle=\lim_{N\to\infty}\sum_{j=0}^NT^j\sum_{n=j}^Nk^{\alpha}(n-j)\sum_{l=n}^{\infty}k^{-\alpha}(l-n)\widehat{\mathfrak{f}}(l)\\
    \displaystyle=\lim_{N\to\infty}\sum_{j=0}^NT^j\left(\widehat{\mathfrak{f}}(j)+\sum_{l=N+1}^{\infty}\widehat{\mathfrak{f}}(l)\sum_{n=j}^Nk^{-\alpha}(l-n)k^{\alpha}(n-j)\right)\\
    \displaystyle=\lim_{N\to\infty}\left(\sum_{j=0}^NT^j\widehat{\mathfrak{f}}(j)+\sum_{l=N+1}^{\infty}\widehat{\mathfrak{f}}(l)\sum_{n=0}^Nk^{-\alpha}(l-n)\Delta^{-\alpha}\mathcal{T}(n)\right)\\
    \displaystyle=\sum_{j=0}^\infty\widehat{\mathfrak{f}}(j)T^j,
    \end{array}\end{equation*}
    where we have applied that $k^{\alpha+1}$ is increasing, $\sum_{n=0}^N|k^{-\alpha}(l-n)|\leq\sum_{n=0}^\infty|k^{-\alpha}(n)|\leq C_{\alpha},$ and $$\sum_{l=N+1}^{\infty}|\widehat{\mathfrak{f}}(l)|k^{\alpha+1}(N)\leq \sum_{l=N+1}^{\infty}|\widehat{\mathfrak{f}}(l)|k^{\alpha+1}(l)\to 0$$ as $N\to\infty.$
\end{itemize}
}
\end{remark}

%The proof of next theorem is inspired in the proof of the Katznelson-Tzafriri Theorem for $C_0$-semigroups given in \cite[p.84]{Ch-To}.

\begin{theorem}\label{main} Let $\alpha> 0,$ $T\in\mathcal{B}(X)$ be a $(C,\alpha)$-bounded operator and $\mathfrak{f}\in A_+^{\alpha}(\mathbb{T})$ be of spectral synthesis in $A^{\alpha}(\mathbb{T})$ with respect to $\sigma(T)\cap \mathbb{T}.$ Then $$\displaystyle\lim_{n\to\infty}\lVert M_T^{\alpha}(n)\theta_{\alpha}(\widehat{\mathfrak{f}}) \rVert=0.$$
\end{theorem}
\begin{proof}

Let $\mathfrak{f}$ be in $A^{\alpha}_+(\mathbb{T})$ of spectral synthesis in $A^{\alpha}(\mathbb{T})$ with respect to $\sigma(T)\cap \mathbb{T},$ that is, for $\varepsilon>0$ there exists  $\mathfrak{f}_{\varepsilon}\in A^{\alpha}(\mathbb{T})$ such that $\lVert \mathfrak{f}-\mathfrak{f}_{\varepsilon} \rVert_{A^{\alpha}(\mathbb{T})}<\varepsilon$ with $\mathfrak{f}_{\varepsilon}= 0$ in a neighborhood $F$ of $\sigma(T)\cap \mathbb{T}.$

Let $(h_n^{\alpha}(j))_{j\in\Z}$ for each $n\in\N_0$ given by \begin{equation*}
h_n^{\alpha}(j):=\left\{\begin{array}{ll}
k^{\alpha}(n-j),&0\leq j\leq n \\
0,& otherwise,
\end{array} \right.
\end{equation*}the natural extension to $\Z$ of the sequences in $\N_0$ defined in \cite[Example 2.5(ii)]{Abadias}. Then note that \begin{equation*}(3.1)\begin{array}{l}
\Delta^{-\alpha} \displaystyle\mathcal{T}(n)\theta_{\alpha}(\widehat{\mathfrak{f}})=\theta_{\alpha}(h_n^{\alpha})\theta_{\alpha}(\widehat{\mathfrak{f}})=\theta_{\alpha}(h_n^{\alpha}*\widehat{\mathfrak{f}})=\displaystyle\sum_{j=0}^{\infty}W_+^{\alpha}(h_n^{\alpha}*\widehat{\mathfrak{f}})(j)\Delta^{-\alpha} \mathcal{T}(j) \\
=\displaystyle\sum_{j=0}^{\infty}W_+^{\alpha}(h_n^{\alpha}*\widehat{\mathfrak{g}_{\varepsilon}})(j)\Delta^{-\alpha} \mathcal{T}(j)+\displaystyle\sum_{j=0}^{\infty}W_+^{\alpha}(h_n^{\alpha}*\widehat{\mathfrak{f}_{\varepsilon}})(j)\Delta^{-\alpha} \mathcal{T}(j),
\end{array}\end{equation*} where we have applied \cite[Theorem 3.5]{Abadias} and $\mathfrak{g}_{\varepsilon}:=\mathfrak{f}-\mathfrak{f}_{\varepsilon}.$  For convenience we write $f(n)=\widehat{\mathfrak{f}}(n)$ for $n\in\N_0,$ $f_{\varepsilon}(n)=\widehat{\mathfrak{f}_{\varepsilon}}(n)$ and $g_{\varepsilon}(n)=\widehat{\mathfrak{g}_{\varepsilon}}(n)=f(n)-f_{\varepsilon}(n)$ for $n\in\Z$ (note that we suppose that $f(n)=0$ for $n<0$ as it is mentioned in the introduction).

On the one hand, we take the first summand. Then using Lemma \ref{lemma2.9}, $W_+^{\alpha}(h_n^{\alpha})=e_n$ (\cite[Example 2.5 (ii)]{Abadias}), \eqref{eq2.7} and Fubini's Theorem we get that

\begin{displaymath}\small{\begin{array}{l}
\displaystyle\sum_{j=0}^{\infty}W_+^{\alpha}(h_n^{\alpha}*g_{\varepsilon})(j)\Delta^{-\alpha} \mathcal{T}(j)=\displaystyle\sum_{j=0}^{n-1}g_{\varepsilon}(j-n)\Delta^{-\alpha} \mathcal{T}(j) \\
+\biggl(\displaystyle\sum_{j=n}^{\infty}\sum_{p=j-n}^j -\displaystyle\sum_{j=0}^{n-1}\sum_{p=j+1}^{\infty} \biggr)k^{\alpha}(p-j+n)W_+^{\alpha}g_{\varepsilon}(p)\Delta^{-\alpha} \mathcal{T}(j) \\
=\displaystyle\sum_{j=0}^{n-1}g_{\varepsilon}(j-n)\Delta^{-\alpha} \mathcal{T}(j) \\
+\biggl(\displaystyle\sum_{p=0}^n\sum_{j=n}^{p+n}+\sum_{p=n+1}^{\infty}\displaystyle\sum_{j=p}^{p+n} -\sum_{p=1}^n\displaystyle\sum_{j=0}^{p-1}-\sum_{p=n+1}^{\infty}\displaystyle\sum_{j=0}^{n-1}\biggr)k^{\alpha}(p-j+n)W_+^{\alpha}g_{\varepsilon}(p)\Delta^{-\alpha} \mathcal{T}(j).
\end{array}}
\end{displaymath}

We now obtain that each term above, when divided by $k^{\alpha+1}(n),$ tends to $0$ as $n\to\infty,$ using that $\lVert \Delta^{-\alpha} \mathcal{T}(j)\rVert \leq C k^{\alpha+1}(j)$ for $j\in\N_0,$ $k^{\alpha+1}(j)$ is increasing as function of $j$ for $\alpha>0,$ the semigroup property of the kernel $k^{\alpha}$ and \eqref{eq2.5}. By Theorem \ref{algebras}, the first term \begin{displaymath}\begin{array}{l}
\displaystyle\frac{1}{k^{\alpha+1}(n)}\sum_{j=0}^{n-1}|g_{\varepsilon}(j-n)| \lVert \Delta^{-\alpha} \mathcal{T}(j)\rVert \\
\displaystyle\leq C \sum_{j=0}^{n-1}|g_{\varepsilon}(j-n)| \leq C \lVert \mathfrak{g}_{\varepsilon} \rVert_{A(\mathbb{T})}\leq C
\lVert \mathfrak{g}_{\varepsilon} \rVert_{A^{\alpha}(\mathbb{T})}  <C \varepsilon.
\end{array}
\end{displaymath}
The second, \begin{displaymath}\begin{array}{l}
\displaystyle\frac{1}{k^{\alpha+1}(n)}\sum_{p=0}^n|W_+^{\alpha}g_{\varepsilon}(p)|\sum_{j=n}^{p+n}k^{\alpha}(p-j+n)\lVert\Delta^{-\alpha} \mathcal{T}(j)\rVert \\
\leq\displaystyle C\sum_{p=0}^n|W_+^{\alpha}g_{\varepsilon}(p)|\frac{k^{\alpha+1}(p+n)}{k^{\alpha+1}(n)}\sum_{j=n}^{p+n}k^{\alpha}(p-j+n) \\
\displaystyle =C\sum_{p=0}^n|W_+^{\alpha}g_{\varepsilon}(p)|\frac{k^{\alpha+1}(p+n)}{k^{\alpha+1}(n)}k^{\alpha+1}(p) \leq C\sum_{p=0}^n|W_+^{\alpha}g_{\varepsilon}(p)|\frac{k^{\alpha+1}(2n)}{k^{\alpha+1}(n)}k^{\alpha+1}(p) \\
\displaystyle \leq  C_{\alpha}\sum_{p=0}^n|W_+^{\alpha}g_{\varepsilon}(p)|k^{\alpha+1}(p)\leq  C_{\alpha}\lVert \mathfrak{g}_{\varepsilon}\rVert_{A^{\alpha}(\mathbb{T})}< C_{\alpha}\varepsilon.
\end{array}
\end{displaymath}
The third term, \begin{displaymath}\begin{array}{l}
\displaystyle\frac{1}{k^{\alpha+1}(n)}\sum_{p=n+1}^{\infty}|W_+^{\alpha}g_{\varepsilon}(p)|\sum_{j=p}^{p+n}k^{\alpha}(p-j+n)\lVert\Delta^{-\alpha} \mathcal{T}(j)\rVert \\
\leq\displaystyle C\sum_{p=n+1}^{\infty}|W_+^{\alpha}g_{\varepsilon}(p)|\frac{k^{\alpha+1}(p+n)}{k^{\alpha+1}(n)}\sum_{j=p}^{p+n}k^{\alpha}(p-j+n) \\
\displaystyle =C\sum_{p=n+1}^{\infty}|W_+^{\alpha}g_{\varepsilon}(p)|k^{\alpha+1}(p+n)
\displaystyle \leq  C_{\alpha}\sum_{p=n+1}^{\infty}|W_+^{\alpha}g_{\varepsilon}(p)|k^{\alpha+1}(p)< C_{\alpha}\varepsilon,
\end{array}
\end{displaymath}
the fourth \begin{displaymath}\begin{array}{l}
\displaystyle\frac{1}{k^{\alpha+1}(n)}\sum_{p=1}^n|W_+^{\alpha}g_{\varepsilon}(p)|\sum_{j=0}^{p-1}k^{\alpha}(p-j+n)\lVert\Delta^{-\alpha} \mathcal{T}(j)\rVert \\
\leq\displaystyle C\sum_{p=1}^n|W_+^{\alpha}g_{\varepsilon}(p)|\frac{k^{\alpha+1}(p)}{k^{\alpha+1}(n)}\sum_{j=0}^{p-1}k^{\alpha}(p-j+n) \\
\leq\displaystyle C\sum_{p=1}^n|W_+^{\alpha}g_{\varepsilon}(p)|\frac{k^{\alpha+1}(p)}{k^{\alpha+1}(n)}\sum_{j=0}^{p+n}k^{\alpha}(p-j+n)\\
= \displaystyle C\sum_{p=1}^n|W_+^{\alpha}g_{\varepsilon}(p)|\frac{k^{\alpha+1}(p)}{k^{\alpha+1}(n)}k^{\alpha+1}(p+n) \\
\displaystyle \leq C_{\alpha}\sum_{p=1}^n|W_+^{\alpha}g_{\varepsilon}(p)|k^{\alpha+1}(p)<C_{\alpha}\varepsilon,
\end{array}
\end{displaymath}
and the fifth \begin{displaymath}\begin{array}{l}
\displaystyle\frac{1}{k^{\alpha+1}(n)}\sum_{p=n+1}^{\infty}|W_+^{\alpha}g_{\varepsilon}(p)|\sum_{j=0}^{n-1}k^{\alpha}(p-j+n)\lVert\Delta^{-\alpha} \mathcal{T}(j)\rVert\\
\leq\displaystyle C\sum_{p=n+1}^{\infty}|W_+^{\alpha}g_{\varepsilon}(p)|\sum_{j=0}^{n-1}k^{\alpha}(p-j+n) \\
\leq\displaystyle C\sum_{p=n+1}^{\infty}|W_+^{\alpha}g_{\varepsilon}(p)|\sum_{j=0}^{p+n}k^{\alpha}(p-j+n)= \displaystyle C\sum_{p=n+1}^{\infty}|W_+^{\alpha}g_{\varepsilon}(p)|k^{\alpha+1}(p+n) \\
\displaystyle \leq  C_{\alpha}\sum_{p=n+1}^{\infty}|W_+^{\alpha}g_{\varepsilon}(p)|k^{\alpha+1}(p)<C_{\alpha}\varepsilon.
\end{array}
\end{displaymath}

On the other hand, for the second term in (3.1), we have to prove that $$\displaystyle\lim_{n\to\infty}\frac{1}{k^{\alpha+1}(n)}\displaystyle\sum_{j=0}^{\infty}W_+^{\alpha}(h_n^{\alpha}*f_{\varepsilon})(j)\Delta^{-\alpha} \mathcal{T}(j)=0.$$ It is known that $\displaystyle(\lambda-T)^{-1}=\biggl(\frac{\lambda-1}{\lambda}\biggr)^{\alpha}\sum_{n=0}^{\infty}\lambda^{-n-1}\Delta^{-\alpha} \mathcal{T}(n),$ for $|\lambda|>1,$ see \cite[Theorem 4.11 (iii)]{Abadias}. Note that $h_n^{\alpha}*f_{\varepsilon}\in \tau^{\alpha}(|n|^{\alpha}),$ then, if $m=[\alpha]+1,$ we get
\begin{displaymath}
\begin{array}{l}
\displaystyle\sum_{j=-\infty}^{\infty}W_+^{\alpha}(h_{n}^{\alpha}*f_{\varepsilon})(-j)e^{ijt}=\sum_{j=-\infty}^{\infty}W_+^{\alpha}(h_{n}^{\alpha}*f_{\varepsilon})(j)e^{-ijt}\\
%=\displaystyle\sum_{j=-\infty}^{\infty}W^mW_+^{-(m-\alpha)}(h_{n}^{\alpha}*f_{\varepsilon})(j)e^{-ijt}\\
=\displaystyle\lim_{\lambda\to 1^+}\biggl( \displaystyle\sum_{j=0}^{\infty}W_+^mW_+^{-(m-\alpha)}(h_{n}^{\alpha}*f_{\varepsilon})(j)(\lambda^{-1} e^{-it})^j\\
+ \displaystyle\sum_{j=-\infty}^{-1}W_+^mW_+^{-(m-\alpha)}(h_{n}^{\alpha}*f_{\varepsilon})(j)(\lambda e^{-it})^j\biggr) \\
=\displaystyle\sum_{l=0}^{m}(-1)^{l} \binom{m}{l}e^{itl}\lim_{\lambda\to 1^+}\biggl(  \displaystyle\sum_{v=l}^{\infty}W_+^{-(m-\alpha)}(h_{n}^{\alpha}*f_{\varepsilon})(v)(\lambda^{-1} e^{-it})^v\\
+ \displaystyle\sum_{v=-\infty}^{l-1}W_+^{-(m-\alpha)}(h_{n}^{\alpha}*f_{\varepsilon})(v)(\lambda e^{-it})^v \biggr)\\
=\displaystyle (1-e^{it})^m\lim_{\lambda\to 1^+}\biggl(  \sum_{u=l}^{\infty}\sum_{v=l}^{u}k^{m-\alpha}(u-v)(\lambda^{-1} e^{-it})^v(h_{n}^{\alpha}*f_{\varepsilon})(u)\\
\displaystyle+ \sum_{u=-\infty}^{l-1}\sum_{v=-\infty}^{u}k^{m-\alpha}(u-v)(\lambda e^{-it})^v(h_{n}^{\alpha}*f_{\varepsilon})(u)
\\
\displaystyle+\sum_{u=l}^{\infty}\sum_{v=-\infty}^{l-1}k^{m-\alpha}(u-v)(\lambda e^{-it})^v(h_{n}^{\alpha}*f_{\varepsilon})(u) \biggr).
\end{array}
\end{displaymath}
Now, using that $$\displaystyle\lim_{\lambda\to 1^+}\sum_{j=0}^{\infty}k^{m-\alpha}(j)(\lambda e^{-it})^{-j}=\frac{1}{(1-e^{it})^{m-\alpha}},\quad t\neq 2\pi\Z,\ 0<m-\alpha<1,$$ see \cite[Section 4]{Abadias}, we have for $t\neq 2\pi\Z$ that
\begin{displaymath}
\begin{array}{l}
\displaystyle\sum_{j=-\infty}^{\infty}W_+^{\alpha}(h_{n}^{\alpha}*f_{\varepsilon})(-j)e^{ijt}\\
=\displaystyle (1-e^{it})^m\biggl(\sum_{u=l}^{\infty}(h_{n}^{\alpha}*f_{\varepsilon})(u)\lim_{\lambda\to 1^+}\biggl(\sum_{v=l}^{u}+ \sum_{v=-\infty}^{l-1}\biggr)k^{m-\alpha}(u-v)(\lambda e^{-it})^v \\
\displaystyle +\sum_{u=-\infty}^{l-1}(h_{n}^{\alpha}*f_{\varepsilon})(u)\lim_{\lambda\to 1^+}\sum_{v=-\infty}^{u}k^{m-\alpha}(u-v)(\lambda e^{-it})^v   \biggr)\\
=\displaystyle(1-e^{it})^{\alpha}\sum_{u=-\infty}^{\infty}(h_{n}^{\alpha}*f_{\varepsilon})(u)e^{-itu}=\displaystyle(1-e^{it})^{\alpha}\mathfrak{f}_{\varepsilon}(-t)\sum_{j=0}^n k^{\alpha}(n-j)e^{-ijt},
\end{array}
\end{displaymath}

If we define $\Delta^{-\alpha} \mathcal{T}(n)=0$ for $n<0,$ note that the operator-valued sequence $(\lambda^{-(j+1)}\Delta^{-\alpha} \mathcal{T}(j))_{j\in\Z}$ for $|\lambda|>1$ is summable. Then Parseval's identity implies that  \begin{displaymath}\begin{array}{l}
\displaystyle\sum_{j=0}^{\infty}W_+^{\alpha}(h_n^{\alpha}*f_{\varepsilon})(j)\Delta^{-\alpha} \mathcal{T}(j)=\lim_{\lambda\to 1^+}\sum_{j=0}^{\infty}W_+^{\alpha}(h_n^{\alpha}*f_{\varepsilon})(j)\lambda^{-(j+1)}\Delta^{-\alpha} \mathcal{T}(j) \\
=\displaystyle\frac{1}{2\pi}\int_0^{2\pi}\mathfrak{f}_{\varepsilon}(-t)\biggl(\sum_{j=0}^{n}k^{\alpha}(n-j)e^{-ijt} \biggr) e^{-it}(e^{-it}-T)^{-1}\,dt\\
=\displaystyle\sum_{j=0}^{n}k^{\alpha}(n-j)\widehat{G}(j),
\end{array}
\end{displaymath}
where $G(t)=e^{-it}\mathfrak{f}_{\varepsilon}(-t)(e^{-it}-T)^{-1}.$  Applying the Riemann-Lebesgue Lemma we get that for all $\delta>0$ there exists a $n_0\in\N$ such that $\lVert \widehat{G}(j)\rVert<\delta$ for all $|j|\geq n_0.$ Then \begin{displaymath}\begin{array}{l}
\displaystyle\frac{1}{k^{\alpha+1}(n)}\lVert  \sum_{j=0}^{n}k^{\alpha}(n-j)\widehat{G}(j)\rVert\leq\frac{1}{k^{\alpha+1}(n)}\biggl(\displaystyle\sum_{j=0}^{n-n_0}+\sum_{j=n-n_0+1}^n\biggr)k^{\alpha}(j)\lVert \widehat{G}(n-j)\rVert \\ \\
\leq\delta +\displaystyle\sum_{j=n-n_0+1}^n \frac{\alpha}{(\alpha+j)}\lVert \widehat{G}(n-j)\rVert\leq\delta+ \frac{\lVert \widehat{G}\rVert_{\infty} (n_0-1)}{\alpha+n-n_0+1},
\end{array}\end{displaymath}
where we have applied that $k^{\alpha+1}(j)$ is increasing as function of $j,$ and $\lVert \widehat{G}\rVert_{\infty}=\sup_{j\geq 0}\lVert \widehat{G}(j)\rVert.$ Taking $n\to\infty$ we get the result.
\end{proof}

\begin{remark}{\rm
Parseval's identity for the product of a scalar-valued function and a vector-valued function, and the Riemann-Lebesgue Lemma for a vector-valued function can be proved by applying linear functionals, and using the scalar-valued results and the Hahn-Banach Theorem. The first reference of these results is \cite{Bochner}. The analogous results for the continuous case are in \cite[Theorem 1.8.1]{ABHN}.
}
\end{remark}

\begin{remark} {\rm
When $T$ is a power-bounded operator, the proof of Theorem \ref{main} gives a short and alternative  proof of the  Katznelson-Tzafriri theorem (\cite[Theorem 5]{Katznelson}), as we show  in the following lines:

Let $\mathfrak{f}$ be in $A_+(\mathbb{T})$ of spectral synthesis in $A(\mathbb{T})$ with respect to $\sigma(T)\cap \mathbb{T},$ that is, for $\varepsilon>0$ there exists  $\mathfrak{f}_{\varepsilon}\in A(\mathbb{T})$ such that $\lVert \mathfrak{f}-\mathfrak{f}_{\varepsilon} \rVert_{A(\mathbb{T})}<\varepsilon$ with $\mathfrak{f}_{\varepsilon}= 0$ in a neighborhood $F$ of $\sigma(T)\cap \mathbb{T}.$ We denote by $(\mathcal{T}(n))_{n\in\Z}$ the family of operators given by $\mathcal{T}(n)=T^n$ for $n\in\N_0$ and $\mathcal{T}(n)=0$ for $n<0.$ Then it is clear that $$\lVert \displaystyle\sum_{j=-\infty}^{\infty}\widehat{\mathfrak{f}_{\varepsilon}}(j)\mathcal{T}(n+j)-T^n\theta(\widehat{\mathfrak{f}}) \rVert<C \varepsilon,$$ since $\lVert T^n\rVert\leq C$ for all $n\in\N_0.$ Now, using Parseval's identity, we get \begin{eqnarray*}
\displaystyle\sum_{j=-\infty}^{\infty}\widehat{\mathfrak{f}_{\varepsilon}}(j)\mathcal{T}(n+j)&=& \displaystyle\lim_{\lambda\to1^+}\sum_{j=-\infty}^{\infty}\widehat{\mathfrak{f}_{\varepsilon}}(j)\lambda^{-(n+j+1)}\mathcal{T}(n+j) \\
&=&\displaystyle\lim_{\lambda\to1^+}\frac{1}{2\pi}\int_{0}^{2\pi}e^{-it(n+1)}\mathfrak{f}_{\varepsilon}(-t)(\lambda e^{-it}-T)^{-1}\,dt \\
&=&\frac{1}{2\pi}\int_{0}^{2\pi}e^{-it(n+1)}\mathfrak{f}_{\varepsilon}(-t)(e^{-it}-T)^{-1}\,dt,
\end{eqnarray*}
which converges to 0 by Riemann-Lebesgue Lemma, and we conclude the proof.}

\end{remark}

\section{Ergodic applications}
\setcounter{theorem}{0}
\setcounter{equation}{0}

Several authors have investigated the connections between the stability of the Ces\`aro mean differences of size $n$ and $n+1,$ that is, \begin{equation}\label{meanDifferences}\displaystyle\lim_{n\to\infty}\lVert M_T^{\alpha}(n+1)\,-\,M_T^{\alpha}(n)\rVert=0,\end{equation}
and spectral conditions for $(C, \alpha)$-bounded operators $T\in\mathcal{B}(X),$ see \cite{Su-Ze13} and references therein. We can not get \eqref{meanDifferences} using directly Theorem \ref{main} because this problem is equivalent to find a sequence $f\in\tau^{\alpha}(n^{\alpha})$ such that the identity $$\frac{1}{k^{\alpha+1}(n)}(h_n^{\alpha}*f)=\frac{1}{k^{\alpha+1}(n)}h_n^{\alpha}-\frac{1}{k^{\alpha+1}(n+1)}h_{n+1}^{\alpha}$$ holds for all $n\in\N_0,$ which has not solution. However the following theorem shows how using Theorem \ref{main} and other techniques we get the desired result, which is a consequence of \cite[Theorem 2.2(ii) and Theorem 3.1(i)]{Su-Ze13} for the case $\alpha\in\N=\{1,2,\ldots\}.$

%However the following theorem shows it is possible to achieve a result in the same direction. We get a stability result of the difference between the Ces\`aro mean of size $n$ of $T$ and the Ces\`aro mean of size $n$ of the semigroup $\mathcal{T}_1(j):=T^{j+1},$ for $j\in\N_0.$

\begin{theorem}\label{ergodicity1} Let $\alpha> 0$ and $T\in\mathcal{B}(X)$ be a $(C,\alpha)$-bounded operator such that $\sigma(T)\cap \mathbb{T}\subseteq\{1\}.$ Then $$\displaystyle\lim_{n\to\infty}\lVert M_T^{\alpha}(n+1)\,-\,M_T^{\alpha}(n)\rVert=0.$$
\end{theorem}
\begin{proof}
Observe that if $\sigma(T)\cap \mathbb{T}=\emptyset,$ then $r(T)<1$ by Lemma \ref{radius}, and therefore $\lVert T^n\rVert\to 0$ exponentially; in particular $T$ is power-bounded. So, we shall prove the result when $\sigma(T)\cap \mathbb{T}=\{1\}.$

First we suppose that $\alpha\geq 1.$ Then using the relation $$\frac{n+\alpha+1}{n+1}M_T^{\alpha}(n+1)-M_T^{\alpha}(n)=\frac{\alpha}{n+1}M_T^{\alpha-1}(n+1),\quad n\in\N_0,$$ which is easy to get from the definition of Ces\`aro mean of order $\alpha,$ we can write $$ M_T^{\alpha}(n+1)\,-\,M_T^{\alpha}(n)=\frac{\alpha}{n+1}(M_T^{\alpha-1}(n+1)-I)+\frac{\alpha}{n+1}(I-M_T^{\alpha}(n+1)).$$ Using the identity $$M_T^{\alpha}(n)(T-I)=\frac{\alpha}{n+1}(M_T^{\alpha-1}(n+1)-I), \quad n\in\N_0,$$ which can easily be obtained from the definition of Ces\`aro mean of order $\alpha,$ and applying Theorem \ref{main} to the function $\mathfrak{f}(t)=e^{it}-1$ we get that the first summand goes to zero when $n\to\infty.$ On the other hand, the second summand goes to zero when $n\to \infty$ since $T$ is a $(C,\alpha)$-bounded operator.

Now let $0<\alpha<1.$ Using that $k^{\alpha}=k^{-(1-\alpha)}*k^1,$ we write $$M_T^{\alpha}(n)=\frac{1}{k^{\alpha+1}(n)}\Delta^{-\alpha}\mathcal{T}(n)=\frac{1}{k^{\alpha+1}(n)}(k^{-(1-\alpha)}*\Delta^{-1}\mathcal{T})(n).$$ So we can write \begin{displaymath}\begin{array}{l}
\displaystyle M_T^{\alpha}(n+1)\,-\,M_T^{\alpha}(n)=\frac{k^{-(1-\alpha)}(n+1)}{k^{\alpha+1}(n+1)}I\\
+\displaystyle\sum_{j=0}^n k^{-(1-\alpha)}(n-j)\biggl(\frac{\Delta^{-1}\mathcal{T}(j+1)}{k^{\alpha+1}(n+1)}-\frac{\Delta^{-1}\mathcal{T}(j)}{k^{\alpha+1}(n)}\biggr)\\
\displaystyle=\frac{k^{-(1-\alpha)}(n+1)}{k^{\alpha+1}(n+1)}I+\frac{n+1}{(n+\alpha+1)k^{\alpha+1}(n)}\displaystyle\sum_{j=0}^n k^{-(1-\alpha)}(n-j)T^{j+1} \\
\displaystyle-\frac{\alpha}{(n+\alpha+1)k^{\alpha+1}(n)}\displaystyle\sum_{j=0}^n k^{-(1-\alpha)}(n-j)\Delta^{-1}\mathcal{T}(j),
\end{array}\end{displaymath}
where we have used that $$\displaystyle\frac{\Delta^{-1}\mathcal{T}(j+1)}{k^{\alpha+1}(n+1)}-\frac{\Delta^{-1}\mathcal{T}(j)}{k^{\alpha+1}(n)}=\frac{1}{(n+\alpha+1)k^{\alpha+1}(n)}\biggl((n+1)T^{j+1}-\alpha\Delta^{-1}\mathcal{T}(j)\biggr).$$
If we add and subtract the term $$\frac{n+1}{(n+\alpha+1)k^{\alpha+1}(n)}\displaystyle\sum_{j=0}^n k^{-(1-\alpha)}(n-j)I=\frac{(k^{-(1-\alpha)}*k^1)(n)}{k^{\alpha+1}(n+1)}I=\frac{k^{\alpha}(n)}{k^{\alpha+1}(n+1)}I$$ then \begin{displaymath}\begin{array}{l}
\displaystyle M_T^{\alpha}(n+1)\,-\,M_T^{\alpha}(n)=\frac{k^{\alpha}(n+1)}{k^{\alpha+1}(n+1)}I \\
\displaystyle+\frac{n+1}{(n+\alpha+1)k^{\alpha+1}(n)}\displaystyle\sum_{j=0}^n k^{-(1-\alpha)}(n-j)(T^{j+1}-I) -\frac{\alpha}{(n+\alpha+1)}M^{\alpha}_T(n).
\end{array}\end{displaymath}
The first term of the above identity goes to zero when $n\to\infty$ using \eqref{double}. If we apply Theorem \ref{main} we get that the second term goes to zero since  \begin{eqnarray*}
M_T^{\alpha}(n)(T-I)&=&\frac{1}{k^{\alpha+1}(n)}\displaystyle\sum_{j=0}^n k^{-(1-\alpha)}(n-j)\Delta^{-1}\mathcal{T}(j)(T-I)\\
&=&\frac{1}{k^{\alpha+1}(n)}\displaystyle\sum_{j=0}^n k^{-(1-\alpha)}(n-j)(T^{j+1}-I).
\end{eqnarray*}
Finally, the third term goes to zero when $n\to\infty$ because $T$ is a $(C,\alpha)$-bounded operator.
\end{proof}

Observe that under the assumption of Theorem \ref{ergodicity1}, we have that $T$ is also $(C,[\alpha]+1)$-bounded, so by \cite[Theorem 2.2(ii) and Theorem 3.1(i)]{Su-Ze13} we have that $\frac{1}{n}\lVert M_{T}^{[\alpha]}(n)\rVert\to 0,$ as $n\to\infty,$ and $\lVert T^n\rVert=o(n^{[\alpha]+1}).$ The following result extends \cite[Theorem 2.2(ii)]{Su-Ze13} for $\alpha\geq 1.$

\begin{remark}\label{remark4}{\rm
Before stating the theorem, note that \cite[Theorem 3.1(i)]{Su-Ze13} is valid for any $\alpha\geq 1.$ In fact, it is enough to follow the same steps in the proof.
}
\end{remark}

\begin{theorem}\label{ergodicity} Let $\alpha\geq 1$ and $T\in\mathcal{B}(X)$ be a $(C,\alpha)$-bounded operator such that $\sigma(T)\cap \mathbb{T}\subseteq\{1\}.$ Then $$\displaystyle\lVert M_T^{\alpha-1}(n)\rVert=o(n)\text{ and }\lVert T^n\rVert=o(n^{\alpha}),\text{ as }n\to\infty.$$
\end{theorem}
\begin{proof}
By Theorem \ref{ergodicity1} and Remark \ref{remark4} we have that $\lVert M_T^{\alpha-1}(n)\rVert=o(n)$. Now, we suppose that $\alpha>1$ for convenience (for $\alpha=1$ the result is proved). We can write $T^n=(k^{-(\alpha-1)}*\Delta^{-(\alpha-1)} \mathcal{T})(n),$ and we have mentioned in the previous section that the sign of $k^{-(\alpha-1)}(n)$ is $(-1)^{[\alpha]}$ for all $n\geq[\alpha].$ For $n\geq[\alpha],$ note that \begin{equation*}\begin{array}{l}\lVert T^n\rVert\leq \displaystyle\sum_{j=0}^n|k^{-(\alpha-1)}(n-j)|\lVert \Delta^{-(\alpha-1)} \mathcal{T})(j\rVert\\
=(-1)^{[\alpha]}\displaystyle\sum_{j=0}^{n}k^{-(\alpha-1)}(n-j)\lVert \Delta^{-(\alpha-1)} \mathcal{T}(j)\rVert\\
+\displaystyle\sum_{j=n-[\alpha]+1}^{n}(|k^{-(\alpha-1)}(n-j)|-(-1)^{[\alpha]}k^{-(\alpha-1)}(n-j))\lVert \Delta^{-(\alpha-1)} \mathcal{T}(j)\rVert=I+II.\\
\end{array}\end{equation*}
By $\lVert M_T^{\alpha-1}(n)\rVert=o(n)$ we have that $\lVert \Delta^{-(\alpha-1)} \mathcal{T}(n)\rVert\leq Ck^{\alpha+1}(n),$ then $$\frac{|I|}{n^{\alpha}}\leq \frac{Ck^{2}(n)}{n^{\alpha}}\to 0,\quad n\to\infty.$$ Secondly,
\begin{equation*}
\begin{array}{l}
\displaystyle\frac{|II|}{n^{\alpha}}\leq \frac{C_{\alpha}}{n^{\alpha}}\displaystyle\sum_{j=n-[\alpha]+1}^{n}\lVert \Delta^{-(\alpha-1)} \mathcal{T}(j)\rVert\\
\displaystyle = \frac{C_{\alpha}}{n^{\alpha}}\displaystyle\sum_{u=0}^{[\alpha]-1}k^{\alpha}(u+n-[\alpha]+1)\lVert  M_T^{\alpha-1}(u+n-[\alpha]+1)\rVert\to 0,\quad n\to\infty.
%&=&\frac{C_{\alpha}k^{\alpha}(n)}{n^{\alpha}}M_T^{\alpha-1}(n)\to 0,\quad n\to\infty.
\end{array}
\end{equation*}

    %where we have used that $k^{\alpha+1}$ is increasing and $k^{-\alpha}*k^{\alpha+1}=k^1.$ Then $\lVert T^j\rVert=O(j^{\alpha}).$ So, if $\mathfrak{f}$ belongs to the extended Korenblyum subalgebra ($\sum_{j\geq 0}j^\alpha|\widehat{\mathfrak{f}}(j)|<\infty$) we have that $\sum_{j\geq 0}\widehat{\mathfrak{f}}(j)T^j$ converges in operator norm.
\end{proof}

\subsection*{Acknowledgements} The author thanks Ralph Chill, Jos\'{e} E. Gal\'{e}, Pedro J. Miana, Daniel J. Rodriguez, Francisco J. Ruiz and the anonymous referee for pieces of advice, comments and nice ideas that have contributed to improve the paper. This paper was partially written during a research visit at the Technical University of Dresden under the supervision of Ralph Chill. The author has been partially supported by Project MTM2013-42105-P, DGI-FEDER, of the MCYTS, and Project E-64, D.G. Arag\'on.


\begin{thebibliography}{999}

\bibitem[ALMV]{Abadias} L. Abadias, C. Lizama, P. J. Miana and M. P. Velasco, \emph{Ces\`{a}ro sums and algebra homorphisms of  bounded operators}, To appear in Israel J. Math (2016).

\bibitem[AOR]{AOR}  G. R. Allan, A. G. O'Farrell and T. J. Ransford, \emph{A Tauberian theorem arising in operator theory,}  Bull. London Math. Soc. {19} (1987), no. 6, 537--545.
%\bibitem{AbdeljaDual} T. Abdeljawad. {\it Dual identities in fractional difference calculus within Riemann.} Adv. in Diff. Equat., vol.2013, article 36, 2013.

%\bibitem{Abdelja2} T. Abdeljawad and F. M. Atici.  {\it  On the definitions of nabla fractional operators.}. Abstr. Appl. Anal. 2012, Article ID 406757 (2012). Doi:10.1155/2012/406757

%\bibitem{Abdelja} T. Abdeljawad. {\it On Riemann and Caputo fractional diferences.} Comput. Math. Appl. 62 (2011), 1602-1611.

\bibitem[ABHN]{ABHN} W. Arendt, C. J. K. Batty, M. Hieber, F. Neubrander: { \it Vector-valued Laplace transforms and Cauchy problems,} Second edition, Monographs in Mathematics. {\bf 96}, Birkh\"auser (2011).

%\bibitem{AtEl09} F. M. Atici and P. W. Eloe. {\it Initial value problems in discrete fractional calculus.} Proc. Amer. Math. Soc., 137 (3), (2009), 981-989.

%\bibitem{AtSe10} F.M. Atici and S. Seng\"ul. {\it Modeling with fractional difference equations.} J. Math. Anal. Appl., 369 (2010), 1-9.

%\bibitem{Lizama1} S. Calzadillas, C. Lizama and J. G. Mesquita. {\it A unified approach to discrete fractional calculus and applications.} Preprint, 2014.


\bibitem[BV]{Ba-Vu} C. J. K. Batty and Q. P. V\~u, \emph{Stability of strongly continuous representations of abelian semigroups,} Math. Z. {209} (1992), no. 1, 75--88.


\bibitem[B]{Bochner} S. Bochner, \emph{Integration von Funktionen, deren Werte die Elemente eines Vektorraumes sind,} Fund. Math. {20} (1933), 262--276.

\bibitem[CT]{Ch-To} R. Chill and Y. Tomilov, \emph{Stability of operators semigroups: ideas and results,} Perspectives in operator theory, Banach Center Publ. {75} (2007), 71--109.

%\bibitem{cho} W. Chojnacki. {\it A generalization of the Widder-Arendt theorem.} Proc. of the Edingburgh. Math. Soc., 45 (2002), 161-179.


%\bibitem{Cu-Pa03} E. Cuesta and C. Palencia. {\it  A numerical method for an integro-differential equation with memory in Banach spaces: Qualitative properties.} SIAM J. Numer. Anal., 41 (3) (2003), 1232-1241.

\bibitem[D]{De00} Y. Derriennic, \emph{ On the mean ergodic theorem for Ces\`aro bounded operators,} Colloq. Math. {84/85} (2000), 443--455.

\bibitem[DL]{DL} Y. Derriennic and M. Lin, \emph{ Fractional Poisson equations and ergodic theorems for fractional coboundaries,} Israel J. Math. {123} (2001), 93--130.


\bibitem[ED]{Ed04} E. Ed-Dari, \emph{On the $(C,\alpha)$ Ces\'aro bounded operators,} Studia Mathematica {161} (2) (2004), 163--175.

\bibitem[E]{Elaydi} S. Elaydi, \emph{An Introduction to Difference Equations,} Undergraduate Texts in Mathematics. Springer. 3rd. Edition, 2005.

\bibitem[Em]{E2} R. Emilion, \emph{Mean-Bounded operators and
mean ergodic theorems,} J.  Func. Anal. {61} (1985), 1--14.

\bibitem[ESZ]{Es-St_Zo} J. Esterle, E Strouse and F. Zouakia, \emph{Stabilit\'e asymptotique de certains semigroupes d'op\'erateurs et ideaux primaires de $L^1(\R^+)$,} J. Operator Theory  {28} (1992), 203--227.

\bibitem[ESZ2]{Es-St_Zo2} J. Esterle, E Strouse and F. Zouakia, \emph{Theorems of Katznelson-Tzafriri type for contractions,} J. Funct. Anal. {94} (1990), 273--287.

\bibitem[GMM]{GMM} J. E. Gal\'e, M. M. Mart\'inez and P.J. Miana, \emph{Katznelson-Tzafriri type theorem for integrated semigroups,} J. Operator Theory {69} (1) (2013), 59--85.

\bibitem[GM]{GM} J. E. Gal\'e and P.J. Miana, \emph{One-parameter groups of regular quasimultipliers,} J. Funct. Anal. {237} (2006), 1--53.

%\bibitem{GMR1} J. E. Gal\'e, P.J. Miana and J. J. Royo. {\it  Estimates of the Laplace transform on convolution
%Sobolev algebras.}  J. Approx. Theory, 164 (2012), 162--178.

%\bibitem{GMR2} J. E. Gal\'e, P.J. Miana and J. J. Royo. {\it Nyman type theorem in convolution Sobolev algebras.} Rev. Mat. Complut., 25 (2012), 1\^{a}€"19.

\bibitem[GW]{Gale} J. E. Gal\'e and A. Wawrzy\'nczyk, \emph{Standard ideals in weighted algebras of Korenblyum and Wiener types,} Math. Scand. {108} (2011), 291--319.

%\bibitem{Gautschi} W. Gautschi. {\it  Some elementary inequalities relating to the gamma and incomplete gamma function}, J. Math. Phys. 38 (1) (1959), 77--81.


%\bibitem{Go12} C.S. Goodrich. {\it On a first-order semipositone discrete fractional boundary value problem.} Arch. Math. 99 (2012), 509-518.

\bibitem[H]{Hille} E. Hille, \emph{Remarks on ergodic theorems,} Trans. Amer. Math. Soc. {57} (1945), 246--269.

\bibitem[K]{Kat} Y. Katznelson, \emph{An Introduction to Harmonic Analysis,} Wiley, New York 1968.

\bibitem[KT]{Katznelson} Y. Katznelson and L. Tzafriri, \emph{On power bounded operators,} J.  Funct. Anal. {68} (1986), 313--328.

%\bibitem{Ke-Li-Mi} V. Keyantuo, C. Lizama and P.J. Miana, {\it Algebra homomorphisms defined via convoluted semigroups and cosine functions.} J.  Funct. Anal. 257 (2009), 3454--3487.

\bibitem[L]{L} Z. L\'eka, \emph{A Katznelson-Tzafriri type theorem in Hilbert spaces,} Proc. Amer. Math. Soc. {137} (2009), no. 11, 3763--3768.

\bibitem[L2]{L2} Z. L\'eka, \emph{A note on the powers of Ces\`aro bounded operators,} Czechoslovak Math. J. 60 {135} (2010), no. 4, 1091--1100.

\bibitem[LSS]{LSS} Y.-C. Li, R. Sato and S.-Y. Shaw, \emph{Boundedness and growth orders of means of discrete and continuous semigroups of operators,} Studia Mathematica {187} (1) (2008), 1--35.

%\bibitem{Lu86} Ch. Lubich. {\it Discretized fractional calculus. }SIAM J. Math. Anal. 17 (1986) 704\^{a}\"{\i}?`½\"{\i}?`½719.

%\bibitem{MeScMo04} M.M. Meerschaert, H.P. Scheffler and J. Mortensen. {\it Vector Gr\"unwald formula for fractional derivatives.} Fractional Calculus Appl. Anal. 7 (1) (2004) 61\^{a}\"{\i}?`½\"{\i}?`½81.

%\bibitem{MeTa04} M.M. Meerschaert and C. Tadjeran. {\it  Finite difference approximations for fractional advection dispersion flow equations. } J. Comput. Appl. Math. 127 (2004) 65--77.

\bibitem[N]{VanNeer} J. V. Neerven., \emph{The asymptotic behaviour of semigroups of linear operators,} Vol. 88, Operator Theory Advances and Applications, Birkh\"auser.

\bibitem[S]{Sa} R. Sato, \emph{Growth orders of means of discrete semigroups of operators in Banach spaces,} Taiwanese J. Math. {14} (3B) (2010), 1111--1116.

\bibitem[SZ]{Su-Ze13} L. Suciu and J. Zem\'anek, \emph{Growth conditions on Ces\`aro means of higher order,} Acta Sci. Math. (Szeged) {79} (2013), 545--581.

\bibitem[TZ]{To-Ze} Y. Tomilov and J. Zem\'anek, \emph{A new way of constructing examples in operator ergodic theory,} Math. Proc. Camb. Philos. Soc. {137} (2004), 209--225.

\bibitem[V]{Vu1} Q. P. V\~u, \emph{Almost periodic and strongly stable semigroups of operators,} Linear operators (Warsaw, 1994), 401--426, Banach Center Publ., 38, Polish Acad. Sci., Warsaw, 1997.

\bibitem[V2]{Vu} Q. P. V\~u, \emph{Theorems of Katznelson-Tzafriri type for semigroups of operators,} J. Funct. Anal. {119} (1992), 74--84.

\bibitem[Y]{Yo98} T. Yoshimoto, \emph{Uniform and strong ergodic theorems in Banach spaces,} Illinois J. Math. {42} (1998), 525--543; Correction, ibid. {43} (1999), 800--801.

\bibitem[Z]{Zygmund} A. Zygmund, \emph{Trigonometric Series,} 2nd ed. Vols. I, II, Cambridge University
Press, New York, 1959.

\end{thebibliography}
\end{document}